\documentclass[a4paper,12pt]{article}
\usepackage{lineno}
\usepackage{amsfonts}
\usepackage{bbold}
\usepackage{xy}
\usepackage{pstricks}
\usepackage{amssymb}

\begin{document}
%\pagewiselinenumbers
%\linenumberdisplaymath

\newtheorem{theorem}{Theorem}[section]

\newtheorem{corollary}[theorem]{Corollary}
\newtheorem{definition}[theorem]{Definition}
\newtheorem{conjecture}[theorem]{Conjecture}
\newtheorem{question}[theorem]{Question}
\newtheorem{lemma}[theorem]{Lemma}
\newtheorem{proposition}[theorem]{Proposition}
\newtheorem{example}[theorem]{Example}
\newtheorem{problem}[theorem]{Problem}
\newenvironment{proof}{\noindent {\bf
Proof.}}{\rule{3mm}{3mm}\par\medskip}
\newcommand{\remark}{\medskip\par\noindent {\bf Remark.~~}}
\newcommand{\pp}{{\it p.}}
\newcommand{\de}{\em}
 \newcommand{\LMA}{{\it Linear and Multilinear Algebra},  }
\newcommand{\JEC}{{\it Europ. J. Combinatorics},  }
\newcommand{\JCTB}{{\it J. Combin. Theory Ser. B.}, }
\newcommand{\JCT}{{\it J. Combin. Theory}, }
\newcommand{\JGT}{{\it J. Graph Theory}, }
\newcommand{\ComHung}{{\it Combinatorica}, }
\newcommand{\DM}{{\it Discrete Math.}, }
\newcommand{\ARS}{{\it Ars Combin.}, }
\newcommand{\SIAMDM}{{\it SIAM J. Discrete Math.}, }
\newcommand{\SIAMADM}{{\it SIAM J. Algebraic Discrete Methods}, }
\newcommand{\SIAMC}{{\it SIAM J. Comput.}, }
\newcommand{\ConAMS}{{\it Contemp. Math. AMS}, }
\newcommand{\TransAMS}{{\it Trans. Amer. Math. Soc.}, }
\newcommand{\AnDM}{{\it Ann. Discrete Math.}, }
\newcommand{\NBS}{{\it J. Res. Nat. Bur. Standards} {\rm B}, }
\newcommand{\ConNum}{{\it Congr. Numer.}, }
\newcommand{\CJM}{{\it Canad. J. Math.}, }
\newcommand{\JLMS}{{\it J. London Math. Soc.}, }
\newcommand{\PLMS}{{\it Proc. London Math. Soc.}, }
\newcommand{\PAMS}{{\it Proc. Amer. Math. Soc.}, }
\newcommand{\JCMCC}{{\it J. Combin. Math. Combin. Comput.}, }
\newcommand{\GC}{{\it Graphs Combin.}, }

\title{ The Algebraic Connectivity and the Clique Number of Graphs
\thanks{ This work is supported by National Natural Science
Foundation of China (No:10971137 and 11271256), the National Basic Research
Program (973) of China (No.2006CB805900).}}
\author{   Ya-Lei Jin  and Xiao-Dong Zhang \\
{\small Department of Mathematics, and MOE-LSC,}\\
{\small Shanghai Jiao Tong University} \\
{\small  800 Dongchuan road, Shanghai, 200240, P.R. China}\\
{\small Email: zzuedujinyalei@163.com (Y.-L. Jin),  xiaodong@sjtu.edu.cn (X.-D. Zhang)}
 }
\date{}
\maketitle
 \begin{abstract}

This paper investigates some relationship between the algebraic connectivity and the clique number of graphs. We characterize all extremal graphs which have the maximum and minimum the algebraic connectivity among all graphs of order $n$ with the clique number $r$, respectively. In turn, an upper and lower bounds for the clique number of a graph in terms of the algebraic connectivity are obtained. Moreover, a spectral version of the Erd\H{o}s-Stone theorem in terms of the algebraic connectivity of graphs is presented.

 \end{abstract}


\vskip 0.5cm

\section{Introduction}
Throughout this paper,  we only consider  simple graphs.
  Let $G = (V(G),~E(G))$ be a simple graph with vertex set
$V(G)=\{v_1, \dots, v_n\}$ and edge set $E(G)$.  Let $A(G)=(a_{ij})$
be the $(0,1)$ {\it  adjacency matrix} of $G$ with $a_{ij}=1$ for
$v_i\sim v_j$ and $0$ otherwise, where $``\sim"$ stands for the adjacency relation. Moreover, let
$D(G)=diag(d(u), u\in V)$ be the diagonal matrix of vertex degree
$d(u)$ of $G$. Then  $L(G)=D(G)-A(G)$ is called {\it the Laplacian matrix }
 of $G$,
 %The Laplacian matrix $L(G)$ can be
  %viewed as an operator on the space of functions
 %$f:V(G)\rightarrow \mathcal{R}$ which satisfies
 %$$L(G)f(u)=\sum_{v\sim u}(f(u)-f(v)).$$
%We order
 and the eigenvalues of $L(G)$  are denoted by
$\lambda_1(G)\geq\lambda_2(G)\geq\dots\geq\lambda_{n-1}(G)\geq\lambda_n(G)=0$.
In particular,  $\lambda_{n-1}(G)$ of $L(G)$ is called \cite{fiedler1973} the {\it algebraic connectivity} of $G$, denoted by $\alpha(G)$,  which plays a vitally important role in graph theory, because it relates with many fundamental graph properties, such as expansion, quasirandom (for example, see \cite{F. Chung1997} and \cite{Nikiforov2013} and references therein).

  In 1941,Tur\'{a}n \cite{Turan1941} determined the maximal number of edges of a graph $G$ which does not contain a copy of the complete graph $K_{r+1}$, which started the research of the extremal theory of graphs. Let $T_{n,r}$,  called {\it $Tur\acute{a}n$ graph}, be the complete $r$-partite graph of order $n$, and the size of every class of which is $\lceil\frac{n}{r}\rceil$ or $\lfloor\frac{n}{r}\rfloor$.
    \begin{theorem}\cite{Turan1941}\label{turan theorem}
Let $G$ be a graph of order $n$ not containing $K_{r+1}$. Then $e(G)\leq e(T_{n,r})$ with equality holding if and only if $G=T_{n,r}$, where $e(G)$ is the number of edges in $G$.
\end{theorem}
   Erd\H{o}s and Stone \cite{erdos1946} (see also, Erd\H{o}s and Simonovits \cite{erdos1966})  expanded the above results.
Let $\mathcal{H}$ be the set of graphs and  $\psi(\mathcal{H})=min\{\chi(H)|H\in\mathcal{H}\}-1$, where $\chi(H)$ is the chromatic number of $H$.   The well-known Erd\H{o}s-Stone theorem~\cite{bondy2008} can be stated as follows:
\begin{theorem}(\cite{erdos1946, erdos1966})\label{ESS thm}
 Let $ex(n, \mathcal{H})$ be the maximum number of edges of a graph with order $n$ not containing a copy of any graph in $\mathcal{H}$.
If  $\psi(\mathcal{H})>1$, then
\begin{equation}
\lim_{n\rightarrow \infty}\frac{ex(n, \mathcal{H})}{\left(\begin{array}{c}
n\\ 2\end{array}\right)}=1-\frac{1}{\psi(\mathcal{H})}.
\end{equation}
\end{theorem}
The theorem gave an insight into the intrinsic role of the chromatic number in extremal graph theory.   During the past over ten years, the spectral extremal graph theory, which establishes the relationship between graph properties and the eigenvalues of certain matrices associated with graphs, have attracted  a lot of attention. Many classical extremal results have been stated, expanded and improved in spectral statement. For example,
 Nikiforov \cite{nikiforov2007} and Guiduli \cite{guiduli1998} had independently proved a spectral extremal
Tur\'{a}n theorem:
\begin{theorem}( \cite{guiduli1998, nikiforov2007}) \label{spec-turan-thm}  Let $\rho(G)$ be the largest eigenvalues of
the adjacency matrix of $G$ of order $n$ not containing  complete subgraph $K_{r+1}$  as a subgraph. Then $\lambda(G)\le \lambda(T_{n,r})$
with equality if and only if $G=T_{n, r}$.
\end{theorem}
 Further, Sudakov, Szabo and Vu \cite{sudakov2005} generalized an asymptotic generalization of Tur\'{a}n¡¯s theorem.  He, Jin and Zhang \cite{He2009} presented an specral Tur\'{a}n theorem in terms of the signless spectral radius of  graphs.  Nikiforov \cite{nikiforov2009a} gave a spectral analogy  for the Erd\H{o}s-Stone-Bollob\'{a}s theorem.
  Recently,  Nikiforov \cite{nikiforov2011} gave an excellent survey on the topic of spectral extremal graph theory and highlighted some connections and analogies  between extremal graph theory and spectral extremal graph theory.
   Moreover, Aouchiche and Hansen \cite{aouchiche2010} presented a lot of conjectures in the spectral graph theory. The relationships between the Laplacian eigenvalues and  clique number have been investigated by many researchers (\cite{guo2011},\cite{Lu2007},\cite{nikiforov2007}, \cite{nikiforov2009b}, \cite{nikiforov2010}, \cite{nikiforov2010-2}, \cite{wilf1986}, etc.).

   In this paper,   we characterize all extremal graphs which have the maximum and minimum the algebraic connectivity among all graphs of order $n$ with the clique number $r$, respectively. The main results are stated as follows:
   
   \begin{theorem}\label{spec-ESS}
   Let $\alpha(n,\mathcal{H})$ be the largest algebraic connectivity of graphs of order $n$ without containing a copy of any graph $H$ in $\mathcal{H}$. Then
   \begin{equation}\label{spec-ess-f}
   \lim_{n\rightarrow\infty}\frac{\alpha(n,\mathcal{H})}{n}=
   1-\frac{1}{\psi(\mathcal{H})},
     \end{equation}
where   $\psi(\mathcal{H})=min\{\ \chi(H)|\ \ H\in\mathcal{H}\ \}-1$.
 \end{theorem}
               Let $G_1$ and $G_2$ be two disjoint graphs. Then the {\it join}, denoted by $G_1\vee G_2$, of $G_1$ and $G_2$ is obtained from $G_1$ and $G_2$ by adding new edges from each vertex in $G_1$ to every vertex of $G_2$.

   \begin{theorem}\label{largest-cliq}
Let $G$ be a non-complete graph of order $n$ not containing $K_{r+1}$.
 Then
 \begin{equation}\label{spec-tu-1}
\alpha(G)\leq n-\lceil\frac{n}{r}\rceil=\alpha(T_{n,r}),
\end{equation}
where $\lceil a \rceil$ is the least integer no less than $a$.
Moreover, if  $n=kr $ or $n=kr+r-1$, then equality (\ref{spec-tu-1}) holds  if and only if  $G$ is Tur\'{a}n graph $T_{n, r}$.
If $n=kr+t, 0<t<r-1$,  then equality (\ref{spec-tu-1}) holds  if and only if  there exist
   graphs $H_1, \dots, H_t$ of order $k+1$ with no edges and $H$ of order $n-(k+1)t$ not containing $K_{r+1-t}$ such that
 $$G=H_1\vee H_2\dots \vee H_t\vee H$$
 and $\alpha(H)\ge n-(k+1)(t+1)$.

\end{theorem}
\begin{theorem}\label{least-cliq}
Let $G$ be a connected graph with the clique number $r\geq 2$. Then \begin{equation} \label{least-1}
\alpha(G)\ge \alpha(Ki_{n, r}),
\end{equation}
where  $Ki_{n,r}$ is a kite graph of order $n$ which is obtained by adding a pendant path of length $n-r$ to a vertex of $K_r$. Moreover, equality (\ref{least-1})  holds if and only if $G=Ki_{n,r}$.
\end{theorem}

   The rest of this paper is organized as follows. In Section 2, 3 and 4, we present proofs of Theorems~\ref{spec-ESS}, \ref{largest-cliq} and \ref{least-cliq}, respectively.

\section{Proof of Theorem~\ref{spec-ESS}}

In order to prove Theorem~\ref{spec-ESS}, we need
 the following Lemma.
\begin{lemma}\cite{B.B}\label{lem2-1}
Given $c>0$ and an integer $r>1$. Then every graph with $n$ vertices and $\lceil(1-\frac{1}{r}+c)\frac{n^2}{2}\rceil$ edges contains a complete (r+1)-partite graph with each part of size $g(n,r+1,c)$, where $g(n,r+1,c)$ tends to infinite with $n$.
\end{lemma}

\begin{lemma}\label{lem2}
Let $T_{n,r}$ be Tur\'{a}n graph.  Then  $\alpha( T_{n, r})=n-\lceil\frac{n}{r}\rceil$.
\end{lemma}
\begin{proof} Let $ \overline{T_{n,r}}$ be the complement graph of $T_{n, r}$. Since $L(T_{n, r})+L(\overline{T_{n,r}})=nI_n-J$, where $I_n$ is the identity matrix and $J$ is the matrix whose all entries are 1,  we have
 $\alpha(T_{n, r})=n-\lambda_1(\overline{T_{n,r}})=n-\lceil\frac{n}{r}\rceil$.
 Hence the assertion holds.
 \end{proof}
  Now we are ready to prove Theorem~\ref{spec-ESS}.

  {\textbf { Proof }} of Theorem~\ref{spec-ESS}.  Let $ \psi(\mathcal {H})=r\ge 2$.   Then $\chi(H)\ge r+1$ for any $H\in \mathcal{H}$.  Clearly, $T_{n, r}$ is a graph of order $n$ and does not contain a copy of any graph $H$ in $\mathcal{H}$. By Lemma~\ref{lem2},
   $\alpha(T_{n,r})=n-\lceil\frac{n}{r}\rceil$.
%$$\lim_{n\rightarrow\infty}\alpha(T_{n,r})=1-\frac{1}{\psi(\mathcal{H})}.$$
Hence
\begin{equation}\label{th1.3-2}
\underline{\lim}_{n\rightarrow\infty}\frac{\alpha(n,\mathcal{H})}{n}
\ge \lim_{n\rightarrow\infty} \frac{\alpha(T_{n,r})}{n}=
\lim_{n\rightarrow\infty}
\frac{n-\lceil\frac{n}{r}\rceil}{n}=1-\frac{1}{r}=1-\frac{1}{\psi(\mathcal{H})}.
\end{equation}
For any given $c>0$, if $G$ is not complete graph of order $n$ and $\alpha(G)>(1-\frac{1}{r}+c)n$, then  by \cite{fiedler1973},
$$ \frac{2e(G)}{n}\ge \delta(G) \geq \alpha(G).$$
Hence
$$e(G)>(1-\frac{1}{r}+c)\frac{n^2}{2}.$$
By Lemma~\ref{lem2-1}, $G$ contains a complete $(r+1)$-partite graph with size of  each part  $g(n,r+1,c)$, where $g(n,r+1,c)$ tends to infinite with $n$. Therefore

\begin{equation}\label{thm1.3-3}
\overline{\lim_{n\rightarrow\infty}}\ \ \frac{\alpha(n,\mathcal{H})}{n}
\le 1-\frac{1}{r}.
\end{equation}
So  (\ref{th1.3-2}) and (\ref{thm1.3-3})  imply that $$\lim_{n\rightarrow\infty}\frac{\alpha(n,\mathcal{H})}{n}=1-\frac{1}{r}=1-\frac{1}{\psi(\mathcal{H})}.$$
We finish our proof.$\blacksquare$

 \begin{theorem}
Let $k\ge 2, r\ge 3$ be integers and $\varepsilon $ be positive. Then there exists an integer $N$, depending on $\varepsilon$ such that every graph $G$  on $n\ge N$ vertices with at least algebraic connectivity $n-\lceil\frac{n}{r}\rceil+\varepsilon n$ contains a copy of the Tur\'{a}n graph $T_{kr, r}$.
\end{theorem}
\begin{proof}
By Erd\H{o}s-Stone Theorem (pp.318 in \cite{bondy2008}), for $d=\frac{\varepsilon}{2}$, there exists an integer $N$, depending on $k,r,$ and $d$,  such that every graph $H$ on $n\ge N$ vertices with at least $e(T_{n,r})+dn^2$ edges contains a copy of the Tur\'{a}n graph $T_{kr, r}$. Moreover, $\varepsilon\ge\frac{1}{N}$.
Now for any graph $G$ on $n\ge N$ vertices with at least algebraic connectivity $n-\lceil\frac{n}{r}\rceil+\varepsilon n$, we have
$$\frac{2e(G)}{n}\ge \alpha(G)\ge n-\lceil\frac{n}{r}\rceil+\varepsilon n.$$
Hence
$$e(G)\ge \frac{n(n-\lceil\frac{n}{r}\rceil+\varepsilon n)}{2}\ge \frac{1}{2}n^2\frac{r-1}{r}+dn^2\ge e(T_{n,r})+dn^2.$$
Hence $G$ contains a  copy of the Tur\'{a}n graph $T_{kr, r}$.
\end{proof}
\section{Proof of Theorem~\ref{largest-cliq}}

In order to prove Theorem~\ref{largest-cliq}, we first prove several Lemmas.
\begin{lemma}\label{lemma3-1}
Let $G$ be a graph of order $n>r$ which does not contain a complete subgraph $K_{r+1}$. Then
\begin{equation}\label{lemma3-1-1}
\alpha(G)\le n-\lceil\frac{n}{r}\rceil.
\end{equation}
Moreover, if $r| n$, then equality in (\ref{lemma3-1-1}) holds if and only if $G$ is Tur\'{a}n graph $T_{n, r}$.
\end{lemma}

\begin{proof} Since $G$ contains no complete subgraph $K_{r+1}$, by Theorem~\ref{turan theorem}, we have $e(G)\le e(T_{n,r})$. Let $n=kr+t$, $0\le t< r$. If $t=0$, then the minimum degree $\delta(G)$ satisfies
 \begin{equation}\label{lemma3-1-2}
  \delta(G)\le \frac{2e(G)}{n}\le \frac{2e(T_{n,r})}{n}
%=\frac{(n-k-1)(k+1)t+(n-k)k(r-t)}{n}
=
n-\frac{n}{r}- \frac{(r-t)t}{rn}.
\end{equation}
Hence $\delta(G)\le n-\lceil\frac{n}{r}\rceil$ with equality if and only if
$t=0$. Further,
 by \cite{fiedler1973} and $G$ is not complete graph of order $n$,
$$\alpha(G)\le \delta(G)\le n-\lceil\frac{n}{r}\rceil.
$$
Moreover, if $r| n$ and  $\alpha(G)=  n-\lceil\frac{n}{r}\rceil$, then
$\delta(G)=\frac{2e(T_{n,r})}{n}$, which implies $e(G)\ge e(T_{n,r})$. Therefore
$e(G)=e(T_{n,r}).$  By Tur\'{a}n Theorem~\ref{turan theorem}, $G$ must be $T_{n,r}$. Conversely, by Lemma~\ref{lem2}, $\alpha(T_{n,r})=n-\lceil\frac{n}{r}\rceil.$
\end{proof}

\begin{corollary}\label{lower}

Let $G$ be a non-complete graph of order $n$ and $r$ be its clique number. Then
\begin{equation}\label{lower-1}
r\ge\frac{n}{n-\alpha(G)}
\end{equation}
with equality if and only if $G$ is Tur\'{a}n graph $T_{n, r}$ and $r| n$.
\end{corollary}
\begin{proof}
By Lemma~\ref{lemma3-1}, $\frac{n}{r}\le \lceil\frac{n}{r}\rceil\le n-\alpha(G)$, which implies the desired result.
\end{proof}

{\textbf Remark:} by Lemma~\ref{lemma3-1}, the extremal graphs of order $n$ without $K_{r+1}$ which attain the maximal algebraic connectivity is unique if $r|n$.  What about $r\not|\ n$?  In this case, the extremal graphs will become more  complex and difficult to characterize. For example, let $G_1=T_{7,3}$ and $G_2$ be a complete 3-partite graph with the sizes of three partitions being 3,3,1. It is easy to see that $\alpha(G_1)=\alpha(G_2)= n-\lceil\frac{n}{r}\rceil
=4$.  Anyway, we are able to obtain some part characterization after a lemma.

\begin{lemma}(\cite{Kirkland2002})\label{kirk}
Let $G$ be a connected, non-complete, graph of order $n$. The the vertex connectivity $\nu(G)$ is equal to $\alpha(G)$ if and only if $G$ can be the join
$G_1\vee G_2$ of two graphs $G_1$ and $G_2$,
 where  $G_1$ is a disconnected graph of order $n-\nu(G)$ and $G_2$ is
a graph of order $\nu(G)$ with $\alpha(G_2)\ge 2\nu(G)-n$.
\end{lemma}
 \begin{lemma}\label{lemma3.2}
Let $G$ be a graph of order $n$ not containing $K_{r+1}$.
 If  $n=kr+t,0< t<r$ and $\alpha(G)=n-\lceil\frac{n}{r}\rceil$, then there exist
 graphs $H_1, \dots, H_t$ of order $k+1$ with no edges and $H$ of order $n-(k+1)t$ not containing $K_{r+1-t}$ such that
 $$G=H_1\vee H_2\dots \vee H_t\vee H$$
 and $\alpha(H)\ge n-(k+1)(t+1)$.
  \end{lemma}

  \begin{proof}
Since $G$ is not complete graph, by \cite{fiedler1973}, $\delta(G)\geq \alpha(G)=n-k-1$. Let $U=\{v\in V(G)\ |\ d_G(v)=\delta(G)\}$.
    Since $G$ does not contain a complete subgraph $K_{r+1}$ of order $r+1$, we have
$$(n-|U|)(\delta(G)+1)+|U|\delta(G)\le 2e(G)\le 2e(T_{n, r})\le n(n-k)-t(k+1).$$
Hence $n(n-k-1-\delta(G))\ge t(k+1)-|U|$, which implies that $\delta(G)=n-k-1$ and $ |U|\ge t(k+1)$.  So  $\delta(G)= \alpha(G)=n-\lceil\frac{n}{r}\rceil$.
On the other hand, by \cite{fiedler1973} and $G$ is not complete, $\alpha(G)\le \nu(G)\le \delta(G) $, which implies $\alpha(G)=\nu(G)=n-k-1$. Hence by Lemma~\ref{kirk}, there exists a disconnected graph $H_1$ of order $k+1$ and a graph $G_1$ of order $n-k-1$ such that
 $G=H_1\vee G_1$ and $\alpha(G_1)\ge n-2k-2$. Moreover, $\delta(G_1)=\alpha(G_1)= n-2k-2$, since $|U|\geq t(k+1)$.   By Lemma~\ref{kirk}, there exists a disconnected graph $H_2$ of order $k+1$ and a graph $G_2$ of order $n-2(k+1)$, such that $G=H_1\vee H_2 \vee G_2$, where $\alpha(G_2)\geq n-3(k+1)$ and
 $\delta(G_2)\ge n-3(k+1)$.  By repeated use of the above process and Lemma~\ref{kirk}, we can get $G=H_1\vee H_2\dots \vee H_{t'}\vee G_{t'},$ where $t\le t'$, $\alpha(G_{t'})\geq n-(k+1)(t'+1)$ and $\delta(G_{t'})\geq n-(k+1)t'-k$.  Further, there are at least $t$ graphs among $H_1, H_2, \dots, H_{t'}$  having no  edges. In fact, if there are at least $(t'-t+1)$ graphs among  $H_1, H_2, \dots, H_{t'}$  having an edges, then $H_1\vee H_2\dots\vee H_{t'}$ contains a complete graph $K_{2t'-t+1}$. In addition, $n-(k+1)t'=k(r-t')+t-t'=k(r-2t'+t-1)+k(t'-t+1)+t-t'$. Thus
\begin{eqnarray*}
2e(G_{t'}) &\ge & [n-(k+1)t']\delta(G_{t'})\\
 &\ge  &[n-(k+1)t'][n-(k+1)t'-k]\\
 &>& 2e(T_{n-(k+1)t',\ r-2t'+t-1}).
 \end{eqnarray*}
Hence by Theorem~\ref{turan theorem}, $G_{t'}$ contains a complete graph $K_{r-2t'+t}$, which implies that $G$ contains $K_{r+1}$. This contradicts to the condition of Lemma~\ref{lemma3.2}. Without loss of generality,  $H_{1},H_2, \dots, H_t$ have no edges. Let $H=H_{t+1}\vee\dots H_{t'}\vee G_{t'}$.  Then
$ \alpha(H)\ge n-(k+1)(t+1)$.
 \end{proof}
 \begin{corollary}\label{t=r-1}
 Let  $G$ be a graph of order $n$ no containing $K_{r+1}$.
 If  $n=kr+r-1$ and $\alpha(G)=n-\lceil\frac{n}{r}\rceil$, then $G$ is Tur\'{a}n graph $T_{n, r}$.
 \end{corollary}
 \begin{proof}
 By Lemma~\ref{lemma3.2}, there exist  graphs $H_1, \dots, H_{r-1}$ of order $k+1$ having no edges and $H$ of order $n-(k+1)(r-1)=k$ such that $\alpha(H)\ge n-(k+1)r=-1$.
 Since $G$ does not contain $K_{r+1}$, $H$ has no edges (otherwise $H$ contains an edge, and $G$ contains $K_{r+1}$, which is a contradiction). So $G$ is Tur\'{a}n graph $T_{n, r}$.\end{proof}

  Now we are ready to prove Theorem~\ref{largest-cliq}.

 {\bf Proof} of Theorem 1.5.   By Lemmas~\ref{lemma3-1} and \ref{lemma3.2}, and Corollary~\ref{t=r-1}, we only to prove that if $n=kr+t, 0<t<r-1$ and there exist
 empty graphs $H_1, \dots, H_t$ of order $k+1$ and $H$ of order $n-(k+1)t$ not containing $K_{r+1-t}$ such that
 $$G=H_1\vee H_2\dots \vee H_t\vee H$$
 and $\alpha(H)\ge n-(k+1)(t+1)$, then $\alpha(G)=n-\lceil\frac{n}{r}\rceil$.
 In fact, the complement graph $\overline{G}$ of $G$ is disjoint of $\overline{H_1},\dots, \overline{H_t}, \overline{H}$.
  Clearly, the largest Laplacian eigenvalue $\lambda_1(H_i)=k+1$ for $i=1, \dots, t$  and $\lambda_1(H)=n-(k+1)t-\alpha(H)\le k+1$. Hence $\lambda_1(\overline{G})\le k+1$. So $\alpha(G)\ge n-(k+1)=n-\lceil\frac{n}{r}\rceil$. Therefore the assertion holds.$\blacksquare$

\section{Proof of Theorem~\ref{least-cliq}}
In this section, we  investigate some properties of the algebraic connectivity of  connected graphs with given clique number. These properties are used to prove Theorem~\ref{least-cliq}, i.e., characterize all extremal graphs with given  clique number which have the minimal algebraic connectivity.
Before proving the main results in this section, we need to recall some known results.

\begin{lemma}\cite{guo2008}\label{guo1}
Let $G$ be  a graph with two vertices $u$ and $v$  and  two  paths $P : uu_1u_{2}\dots u_k$ and
$Q : vv_1\dots v_l$ of lengths $k, l \ (k, l\geq 1)$. Let $G_{k,l}$ be the graph from $G$ by attached two paths $P, Q$  with $u$ and $v$, respectively. Further, let
\begin{equation}
G'_{k+l}= G_{k,l}- uu_1 + u_1v_l,\ \
G''_{k+l}= G_{k,l}- vv_1 + u_kv_1.
\end{equation}
If $X=(X(v), v\in V(G_{k,l}))^T$ is a Fiedler vector of $G_{k,l }$
 and  $X(u_k)X(v_l) \geq 0$, then
\begin{equation}\label{guo1-1}
\alpha(G_{k,l} ) \geq \min\{ \ \alpha(G'_{k+l}),\ \  \alpha(G''_{k+l})\ \}.
\end{equation}
Moreover,  equality in $(\ref{guo1-1})$  holds if and only if $X(v_k)=X(u_l) = 0$.
\end{lemma}

%In order to present the following results, we need some notations.

\begin{lemma}\cite{guo2010}\label{guo2}
 Let G be a connected graph with  at least two  vertices  and  two  paths $P : uu_1u_2\dots u_k$ and
$Q : uv_1... v_l$ of lengths $k, l\ (k, l\ge 1) $. If $G_{k,l}$  is the graph  obtained from $G$ by attached two paths $P, Q$ at vertex $u$ and
 $G_{k+1, l-1}^{(1)}=G_{k, l}-v_{l-1}v_l+u_kv_l$ (see Fig. $1$) and
$k\ge l\ge 1$, then
\begin{equation}\label{guo2-1}
\alpha(G_{k,l} )\geq \alpha(G_{k+1,l-1}^{(1)} ).
\end{equation}
Moreover, inequality (\ref{guo2-1}) is strict if either $X(v_1)\neq 0$ or $X(u_1)\neq 0$.\\

\begin{picture}(50, 28)
\begin{picture}(50, 28)

\put(30,10){\circle{200}}
\put(28,10){$G$}
\put(40,10){$ u$}
\put(37,10){\circle*{2}}
\put(45,18){\circle*{2}}
\put(55,18){\circle*{2}}
\put(45,20){$ u_1$}
\put(55,20){$ u_k$}

\put(55,2){\circle*{2}}
\put(45,2){\circle*{2}}
\put(45,2){........}
\put(37,10){\line(1,-1){8}}
\put(45,4){$ v_1$}
\put(55,4){$ v_l$}

\put(45,18){........}
\put(37,10){\line(1,1){8}}

\put(28,-8){$G_{k,l}$}
\end{picture}

\begin{picture}(50, 28)

\put(30,10){\circle{200}}
\put(28,10){$G$}
\put(40,10){$ u$}
\put(37,10){\circle*{2}}
\put(45,18){\circle*{2}}
\put(55,18){\circle*{2}}
\put(63,18){\circle*{2}}
\put(45,20){$ u_1$}
\put(55,20){$ u_k$}
\put(63,20){$ v_l$}
\put(55,18){\line(1,0){8}}

\put(55,2){\circle*{2}}
\put(45,2){\circle*{2}}
\put(45,2){........}
\put(37,10){\line(1,-1){8}}
\put(45,4){$ v_1$}
\put(55,4){$ v_{l-1}$}

\put(45,18){........}
\put(37,10){\line(1,1){8}}
\put(28,-8){$G_{k+1,l-1}$}
\put(10,-13){Fig. $1$}
\end{picture}
\end{picture}\\ \\ \\
\end{lemma}

%The following results show that $K(n,r)$ is the graph with minimal algebraic connectivity in the connected graph with clique number $r$.For the case $r=3$,Guo\cite{guo2008} has given the proof.we want to prove for all the cases here. Before the proof of the main theorem,we need some lemmas.

\begin{lemma}\label{sign}
Let $G_{k,l}^{(2)}$ be a graph obtained from the complete graph $K_r, r\ge 3$ with vertex set
$V(K_r)=\{w_1, \dots, w_{r-2}, u, v\}$ by attached two paths $P=uu_1\dots u_k$ and  $Q=vv_1\dots v_l$ $(k, l\ge 1)$ at vertices $u$ and $v$, respectively (see Fig.2). If
$X=(X(w), w\in V(G_{k,l}^{(2)}))^T$ is a Fiedler vector, then $X(w_1)=\dots=X(w_{r-2})^T$ and $X(u_k)X(v_l)<0$. Moreover, if $X(w_1)\ge 0, X(u_k)>0$ and $k\ge 1$, then $X(u_1)>X(u)>X(w_1).$
\center{
\vskip 2cm
\setlength{\unitlength}{0.90mm}
\begin{picture}(20, 20)
\put(14,14){\circle{200}}
\put(10,10){$K_{r-2}$}

\put(-9,25){\circle*{2}}
\put(0,25){\circle*{2}}
\put(5,25){\circle*{2}}
\put(-9,27){$u_k$}
\put(0,27){$u_1$}
\put(5,27){$u$}
\put(-7,25){...}
\put(0,25){\line(1,0){5}}
\put(5,25){\line(1,0){18}}
\put(5,25){\line(1,-2){6}}
\put(5,25){\line(1,-1){12}}

\put(38,25){\circle*{2}}
\put(28,25){\circle*{2}}
\put(23,25){\circle*{2}}
\put(38,27){$v_l$}
\put(28,27){$v_1$}
\put(23,27){$v$}
\put(31,25){...}
\put(23,25){\line(1,0){5}}
\put(23,25){\line(-1,-2){6}}
\put(23,25){\line(-1,-1){12}}
\put(12,0){$G_{k,l}^{(2)}$}
\put(-2,0){Fig.2}
\end{picture}}
\end{lemma}
\begin{proof}
Since $k, l\ge 1, $
by \cite{fiedler1973}, $\alpha(G_{k,l}^{(2)})\le 1$.
By the equations of $L(G_{k,l}^{(2)})X=\alpha(G_{k,l}^{(2)})X$ corresponding to vertices $w_i$ and $w_j$, we have
$$\alpha(G_{k,l}^{(2)})X(w_i)-\alpha(G_{k,l}^{(2)})X(w_j)=r(X(w_i)-X(w_j)).$$
Hence $X(w_i)=X(w_j)$ for $1\le i\neq j\le r-2$.

 Supposed that $X(u_k)X(v_l)\ge 0$.  Assume that
 $X(u_k)\ge 0, X(v_l) \ge 0$ (or $X(u_k)\le 0, X(v_l)\le 0$),  let  $U=\{w \ | \ x(w)\ge 0 \ w\in V(G_{k,l}^{(2)})\ \}$ (or let $U=\{w \ |\ x(w)\le 0 \ w\in V(G_{k,l}^{(2)})\ \}$). Then
by Theorem~3.3 in \cite{fiedler1975},  the induced subgraph $G_{k, l}^{(2)}[U]$ by the vertex set $U$ is connected.  Hence $u_k, \dots, u_1, u, v, v_1, \dots, v_l\in U$. So
$X(u_i)\ge 0, X(v_j)\ge 0,$ $ i=1, \dots, k,, j=1, \dots, l$ and $ X(u)\ge 0, X(v)=0$.  By the equation $L(G_{k, l}^{(2)})X=\alpha(G_{k, l}^{(2)})X$, $\alpha(X(w_1))=2X(w_1)-X(u)-X(v)$, which implies $X(w_1)\ge 0$. Hence $X$ is  nonnegative vector, which contradicts to $X\neq 0$ being orthogonal to the vectors of all 1, since $X$ is a Fiedler vector. Hence $X(u_k)X(v_l)<0$ holds.

Moreover, if  $X(w_1)\ge 0, X(u_k)>0$, then  $w_1, u_k\in U_1=:\{\ w\in V(G)\ | \ X(w)\ge 0\ \}$. By Theorem~3.3 in \cite{fiedler1975}, the induced subgraph of $G_{k, l}^{(2)}$
by $U_1$ is connected. Hence $X(u_i)\ge 0$ for $i=1, \dots, k$ and $X(u)\ge 0$.
By $L(G_{k, l}^{(2)})X=\alpha(G_{k, l}^{(2)})X$,  we have $\alpha(G_{k, l}^{(2)})X(u_k)=X(u_k)-X(u_{k-1}))>0$, which implies
$X(u_{k-1})-X(u_{k-2})>0$ by $\alpha(G_{k, l}^{(2)})X(u_{k-1})=(X(u_{k-1})-X(u_k)) +(X(u_{k-1})-X(u_{k-2}))$.  By the induction method, we are able to prove that
$X(u_i)-X(u_{i-1})>0$ for $i=2, \dots, k$. By $\alpha(G_{k, l}^{(2)}) X(u_1)=(X(u_1)-X(u_2))+(X(u_1)-X(u))$, we have $X(u_1)-X(u)>0$.
By $\alpha(G_{k, l}^{(2)})X(u)= rX(u)-(r-2)X(w_1)-X(u_1)-X(v)$ and $\alpha(G_{k, l}^{(2)})X(w_1)=2X(w_1)-X(u)-X(v)$, we have
$(r-\alpha(G_{k, l}^{(2)}))(X(u)-X(w_1))=X(u_1)-X(u)>0,$  which implies that $X(u)>X(w_1)$.
\end{proof}

\begin{lemma}\label{switch}
Let $G_{k, l}^{(2)}$ be the graph obtained from a complete graph $K_r, r\ge 3$ with vertex set
$V(K_r)=\{w_1, \dots, w_{r-2}, u, v\}$ by attached two paths $P=uu_1\dots u_k$ and  $Q=vv_1\dots v_l$ at vertices $u$ and $v$, respectively (see Fig. 3).
If $G_{k+1,l-1}^{(2)}=G-\{uw_i,1\leq i\leq r-2\}+\{v_1w_i,1\leq i\leq r-2\}$,
$G_{k-1,l+1}^{(2)}=G-\{vw_i,1\leq i\leq r-2\}+\{u_1w_i,1\leq i\leq r-2\}$.
 Then
\begin{equation}\label{switch-1}
\alpha(G_{k, l}^{(2)})>\min \{\ \alpha(G_{k+1,l-1}^{(2)}),\ \alpha(G_{k-1,l+1}^{(2)})\ \}.
\end{equation}
\end{lemma}
\vskip 1cm
\setlength{\unitlength}{0.85mm}
\begin{picture}(50,28)\label{figswitch}
\begin{picture}(30, 20)
\put(24,14){\circle{200}}
\put(20,10){$K_{r-2}$}

\put(1,25){\circle*{2}}
\put(10,25){\circle*{2}}
\put(15,25){\circle*{2}}
\put(1,27){$u_k$}
\put(10,27){$u_1$}
\put(15,27){$u$}
\put(3,25){...}
\put(10,25){\line(1,0){5}}
\put(15,25){\line(1,0){18}}
\put(15,25){\line(1,-2){6}}
\put(15,25){\line(1,-1){12}}

\put(48,25){\circle*{2}}
\put(38,25){\circle*{2}}
\put(33,25){\circle*{2}}
\put(48,27){$v_l$}
\put(38,27){$v_1$}
\put(33,27){$v$}
\put(41,25){...}
\put(33,25){\line(1,0){5}}
\put(33,25){\line(-1,-2){6}}
\put(33,25){\line(-1,-1){12}}
\put(22,0){$G_{k, l}^{(2)}$}
\end{picture}

\begin{picture}(30, 20)
\put(49,14){\circle{200}}
\put(45,10){$K_{r-2}$}

\put(25,25){\circle*{2}}
\put(35,25){\circle*{2}}
\put(40,25){\circle*{2}}
\put(25,27){$u_k$}
\put(35,27){$u$}
\put(40,27){$v$}
\put(28,25){...}
\put(35,25){\line(1,0){5}}
\put(40,25){\line(1,0){18}}
\put(40,25){\line(1,-2){6}}
\put(40,25){\line(1,-1){12}}

\put(73,25){\circle*{2}}
\put(63,25){\circle*{2}}
\put(58,25){\circle*{2}}
\put(73,27){$v_l$}
\put(63,27){$v_2$}
\put(58,27){$v_1$}
\put(66,25){...}
\put(58,25){\line(1,0){5}}
\put(58,25){\line(-1,-2){6}}
\put(58,25){\line(-1,-1){12}}
\put(45,0){$G_{k+1,\ l-1}$}
\end{picture}

\begin{picture}(30, 20)
\put(79,14){\circle{200}}
\put(75,10){$K_{r-2}$}

\put(55,25){\circle*{2}}
\put(65,25){\circle*{2}}
\put(70,25){\circle*{2}}
\put(55,27){$u_k$}
\put(65,27){$u_2$}
\put(70,27){$u_1$}
\put(58,25){...}
\put(65,25){\line(1,0){5}}
\put(70,25){\line(1,0){18}}
\put(70,25){\line(1,-2){6}}
\put(70,25){\line(1,-1){12}}

\put(103,25){\circle*{2}}
\put(93,25){\circle*{2}}
\put(88,25){\circle*{2}}
\put(103,27){$v_l$}
\put(93,27){$v$}
\put(88,27){$u$}
\put(96,25){...}
\put(88,25){\line(1,0){5}}
\put(88,25){\line(-1,-2){6}}
\put(88,25){\line(-1,-1){12}}
\put(75,0){$G_{k-1,\ l+1}$}
\end{picture}
\put(-20,-8){$Fig.3$}
\end{picture}\\ \\

\begin{proof}
 Let $X$ be a Fiedler vector corresponding to $\alpha(G_{k,l}^{(2)})$.
  By Lemma~\ref{sign},  $X(w_1)=\dots =X(w_{r-2}):=a$.
    Without loss of generality, assume $a\geq 0$ (otherwise consider $-X$).
    By Lemma~\ref{sign},  $X(v_l)X(u_k)<0$.  Hence
     consider the following two cases

     {\bf Case 1:} $X(v_l)<0<X(u_k)$.   Denote $V_3= \{u,v,v_1,v_2,\dots,v_k,u_1,u_2,\dots,u_l\}, A=x(u_1)-x(u)$ and $B=x(u)-x(v)$. Further, we consider the following two subcases.

\textbf{Subcase 1.1:} $A\leq B$. Let
\begin{displaymath}
Y(w) = \left\{ \begin{array}{ll}
X(w)-b, & \textrm{ $w\in V_3$}\\
X(w)+\frac{n-r+2}{r-2}b, & \textrm{$w\in  \{w_1, \dots, w_{r-2}\}$ }
\end{array} \right.,
\end{displaymath}
where $b=\frac{r-2}{n}A>0$ by Lemma~\ref{sign}.  It is easy to see that  $Y$ is orthogonal to the vector whose each component is one.
 By $\alpha(G_{k,l}^{(2)})a =\alpha(G_{k,l}^{(2)})X(w_1)=2a-X(u)-X(v)$, we have
\begin{eqnarray*}
&&Y^TL(G_{k-1,l+1}^{(2)})Y-X^TL(G_{k, l}^{(2)})X
\\
&=& (r-2)\{ [(Y(u_1)-a)^2+(Y(u)-a)^2]-[(X(u)-a)^2+(X(v)-a)^2]\}\\
&=&(r-2)\{ [ (X(u_1)-a-A)^2-(X(u)-a)^2] +[(X(u)-a-A)^2-(X(v)-a)^2]\}\\
&=& (r-2)[0+(B-A)(X(u)+X(v)-2a-A)]\\
&=& (r-2)(B-A)(-\alpha(G_{k, l}^{(2)}) a-A)\le 0.
\end{eqnarray*}
In addition, $\sum_{w\in V(G_{k, l}^{(2)})}X(w)^2=1$ and $\sum_{w\in V(G_{k, l}^{(2)})}X(w)=0$, which implies that $\sum_{w\in V_3}X(w)=-(r-2)X(w_1)=-(r-2)a.$  Hence
 \begin{eqnarray*}
&Y^TY&=\sum_{w\in V_3}(X(w)-b)^2+\sum_{w\notin V_3}(X(w)+\frac{n-r+2}{r-2}b)^2\\
&=&1+2b\sum_{w\notin V_3}X(w)+b^2(k+l+2)+\frac{2(n-r+2)}{r-2}b\sum_{w\notin V_3}X(w)\\
&&+(r-2)(\frac{2(n-r+2)}{r-2}b)^2\\
&=& 1+2abn+(k+l+2)b^2+\frac{4(n-r+2)b^2}{r-2}>1.
 \end{eqnarray*}                                                                                                                                                                       Then\begin{eqnarray*}
 \alpha(G_{k-1,l+1}^{(2)})
 &\le & \min_{Z\neq 0, z_1+\dots+ z_n=0}
 \frac{Z^TL(G_{k-1,l+1}^{(2)})Z}{Z^TZ}\le \frac{Y^TL(G_{k-1,l+1}^{(2)})Y}{Y^TY}\\
 &\le & \frac{X^TL(G_{k, l}^{(2)})X}{Y^TY}< \alpha(G_{k, l}^{(2)}).
 \end{eqnarray*}

 {\bf Subcase 1.2: }$A> B$.  Let
 \begin{displaymath}
Y(w) = \left\{ \begin{array}{ll}
X(u)-c+A-B, & {\rm if }\  w=u\\
X(w)-c, & {\rm if}\  w\in V_3\setminus \{u\},\\
X(w)+\frac{(n-r+2)c}{r-2}-\frac{A-B}{r-2},& \textrm{if \ $w\in \{w_1,...,w_{r-2}\}$}
\end{array} \right.,
\end{displaymath}
where  $c=\frac{(r-1)A-B}{n}$. Then $\sum_{w\in V(G)}Y(w)=0$, $Y(u)-Y(w_1)=X(v)-X(w_1)$, $Y(u_1)-Y(w_1)=X(u)-X(w_1)$, $Y(u)-Y(v)=X(u_1)-X(u)$ and $Y(u)-Y(u_1)=X(u)-X(v).$
Hence
\begin{eqnarray*}
&&Y^TL(G_{k-1,l+1}^{(2)})Y-X^TL(G_{k, l}^{(2)})X\\
&=&(r-2)\{(Y(u_1)-Y(w_1))^2+(Y(u)-Y(w_1))^2\}+ (Y(u)-Y(u_1))^2+(Y(u)-Y(v))^2
\\
 &&- (r-2)\{ (X(u)-X(w_1))^2 +(X(v)-X(w_1))^2\}-(X(u)-X(u_1))^2-(X(u)-X(v))^2\\
 &=& 0.
 \end{eqnarray*}
In addition, $\sum_{w\in V(G_{k, l}^{(2)})}X(w)^2=1$ and $\sum_{w\in V_3}X(w)=-(r-2)a$. Hence
\begin{eqnarray*}
Y^TY&=&\sum_{w\in V_3,w\neq u}(X(w)-c)^2+(X(u)-c+A-B)^2+(r-2)(a+\frac{(n-r+2)c}{r-2}-\frac{A-B}{r-2})^2\\
&=&\sum_{w\in V(G_{k, l}^{(2)})}X(w)^2-2c\sum_{w\in V_3}X(w) +(k+l-1)c^2+2X(u)(A-B)+(-c+A-B)^2\\
&&+ 2a(r-2)[\frac{n-r+2}{r-2}c-\frac{A-B}{r-2}]+(r-2)(\frac{n-r+2}{r-2}c-\frac{A-B}{r-2})^2\\
&=& 1+2(r-2)ac+2a(r-2)[\frac{n-r+2}{r-2}c-\frac{A-B}{r-2}]+2X(u)(A-B)+(k+l-1)c^2\\
&&+ (-c+A-B)^2+(r-2)(\frac{n-r+2}{r-2}c-\frac{A-B}{r-2})^2\\
&=& 1+2a(r-2)A+ 2X(u)(A-B)+(k+l-1)c^2+ (-c+A-B)^2\\
&&+(r-2)(\frac{n-r+2}{r-2}c-\frac{A-B}{r-2})^2
\\
&>& 1,
\end{eqnarray*}
since $A>0$, $A-B>0$ and $ X(u)\ge 0$.
Therefore,
\begin{eqnarray*}
\alpha(G_{k-1,l+1}^{(2)}) &\le& \min_{Z\neq 0, z_1+\dots +z_n=0} \frac{Z^TL(G_{k-1,l+1}^{(2)})Z}{Z^TZ}
\\
&\le &
\frac{Y^TL(G_{k-1, l+1}^{(2)})Y}{Y^TY}=\frac{X^TL(G_{k, l}^{(2)})X}{Y^TY}\\
&<& X^TL(G_{k, l}^{(2)})X=\alpha(G_{k, l}^{(2)}).
\end{eqnarray*}
Hence $\alpha(G_{k-1,l+1}^{(2)})<\alpha(G_{k, l}^{(2)}).$

{\bf Case 2:} $X(u_k)<0<X(v_l)$. By the same method, we are able to prove that
$\alpha(G_{k+1,l-1}^{(2)})<\alpha(G_{k, l}^{(2)}).$ The detail is omitted.
Hence the assertion holds.
\end{proof}
\begin{corollary}\label{cor4.5}
Let $G_{k,l}^{(2)}$ be the graph obtained from a complete graph $K_r, r\ge 3$ with vertex set
$V(K_r)=\{w_1, \dots, w_{r-2}, u, v\}$ by attached two paths $P=uu_1\dots u_k$ and  $Q=vv_1\dots v_l$ at vertices $u$ and $v$, respectively. If $k\geq l\ge 1$, then \begin{equation}\label{cor4.5-1}
\alpha(G_{k, l}^{(2)})>\alpha(G_{k+1,l-1}^{(2)}).\end{equation}
\end{corollary}
\begin{proof}
If $k=l$, the assertion follows from  Lemma~\ref{switch} and $G_{k+1,l-1}^{(2)}=G_{k-1, l+1}^{(2)}$.  If $k>l$, suppose that
$\alpha(G_{k,l}^{(2)})>\alpha(G_{k+1, l-1}^{(2)})$ does not hold, i.e.,  $\alpha(G_{k,l}^{(2)})\le \alpha(G_{k+1,l-1}^{(2)})$. Then
by Lemma \ref{switch}, we have
\begin{equation}\label{cor4.5-2}
\alpha(G_{k,l}^{(2)})>\min \{\alpha(G_{k-1,l+1}^{(2)}),\alpha(G_{k+1,  l-1}^{(2)})\}=\alpha(G_{k-1,l+1}^{(2)}).
\end{equation}
By repeated uses of Lemma~\ref{switch} and (\ref{cor4.5-2}), we have
\begin{equation}\label{cor4.5-3}
\alpha(G_{k-1,l+1}^{(2)})>\min \{\alpha(G_{k-2,l+2}^{(2)}),\alpha(G_{k,l}^{(2)})\}=
\alpha(G_{k-2,l+2}^{(2)}).
\end{equation}
  By repeated uses of Lemma \ref{switch}, we obtain  $$\alpha(G_{k, l}^{(2)})>\alpha(G_{k-1,l+1}^{(2)})>
  \alpha(G_{k-2,l+2}^{(2)})>\dots >\alpha(G_{k-(k-l),l+(k-l)}^{(2)})=\alpha(G_{l,k}^{(2)}),$$
  which is a contradiction, since $G_{k,l}^{(2)}=G_{l,k}^{(2)}$.
  So the assertion holds.
\end{proof}

\begin{corollary}\label{cor4.6} Let
$G_{k,l}^{(2)}$ be a graph of order $n$ obtained from the complete graph $K_r, r\ge 3$ with vertex set
$V(K_r)=\{w_1, \dots, w_{r-2}, u, v\}$ by attached two paths $P=uu_1\dots u_k$ and  $Q=vv_1\dots v_l$ at vertices $u$ and $v$, respectively, where $n=r+k+l$.
If $k>0, l>0$, then
\begin{equation}\label{cor2-1}
\alpha(G_{k, l}^{(2)})>\alpha(Ki_{n, r}).
\end{equation}\end{corollary}
\begin{proof}
The assertion follows from the repeated use of Corollary~\ref{cor4.5}.\end{proof}

\begin{lemma}\label{switch2}
Let $G$ be a graph of order $n$ obtained from a complete graph $K_r$ with vertex set $V(K_r)=\{w_1, \dots, w_{r-2}, u, v\}$ and a path $P=u_1u_2\dots u_k$ by joining two edges $uu_1$ and $vu_1$, where $n=r+k$.  If $3\le r\le n-1,$ then
\begin{equation}
\alpha(G)>\alpha(Ki_{n, r}).
\end{equation}

\begin{picture}(50, 20)
\put(50,10){$K_r$}
\put(55,10){\circle{200}}
\put(61,14){\circle*{2}}
\put(61,6){\circle*{2}}
\put(56,14){u}
\put(56,6){v}
\put(69,10){\circle*{2}}
\put(70,6){$u_1$}
\put(78,10){\circle*{2}}
\put(78,6){$u_2$}
\put(81,10){$\ldots$}
\put(88,10){\circle*{2}}
\put(88,6){$u_k$}
\put(70,10){\line(1,0){8}}
\put(61,14){\line(2,-1){8}}
\put(61,6){\line(2,1){8}}
\put(70,-8){$G$}
\put(55,-8){Fig.$4$}

\end{picture}\\

\end{lemma}
\begin{proof}
Let $X=(X(w), w\in V(G))^T$ be a Fiedler vector of $G$.
By $L(G)X=\alpha(G)X$, we have
\begin{equation}\label{switch2-1}
\alpha(G)X(w_i)=(r-1)X(w_i)-\sum_{t=1, t\neq i}^{r-2}X(w_t)-X(u)-X(v),
\end{equation}

\begin{equation}\label{switch2-2}
\alpha(G)X(w_j)=(r-1)X(w_j)-\sum_{t=1, t\neq j}^{r-2}X(w_t)-X(u)-X(v),
\end{equation}
\begin{equation}\label{switch2-3}
\alpha(G)X(u)=rX(u)-\sum_{t=1,}^{r-2}X(w_t)-X(u_1)-X(v),
\end{equation}
\begin{equation}\label{switch2-4}
\alpha(G)X(v)=rX(v)-\sum_{t=1,}^{r-2}X(w_t)-X(u_1)-X(u).
\end{equation}
Subtracting (\ref{switch2-1}) from (\ref{switch2-2}) gets
$\alpha(G)(X(w_i)-X(w_j))=r(X(w_i)-X(w_j))$, which implies $X(w_i)=X(w_j)$ for $i,j=1, \dots, r-2$.  On the other hand, subtracting (\ref{switch2-3}) from (\ref{switch2-4}) gets $\alpha(G)(X(u)-X(v))=(r+1)(X(u)-X(v))$, which implies $X(u)=X(v)$.
Further, by (\ref{switch2-1}),
$X(w_i)=\frac{2X(u)}{2-\alpha(G)}.$
Hence by (\ref{switch2-3}), we have
$$\alpha(G)X(u)=(r-1)X(u)-(r-2)X(w_i)-X(u_1),$$
 which implies
 \begin{equation}\label{switch2-6}
 X(u_1)=\frac{\alpha(G)^2-(r+1)\alpha(G)+2}{2-\alpha(G)}X(u).
 \end{equation}
Then $X(u_1)\neq 0$. Otherwise $X(u)=X(v)=X(w_i)=0$, which implies $X(u_2)=\dots =X(u_k)=0$. It is a contradiction.  Further $X(u)\neq X(u_1)$, since  $\alpha(G)^2-(r+1)\alpha(G)+2\neq 2-\alpha(G)$. Note that $ Ki_{n, r}=G-uu_1$.
Hence
$$X^TL(Ki_{n, r})X=X^TL(G-uu_1)X=X^TL(G)X-(X(u)-X(u_1))^2<\alpha(G).$$
Therefore $\alpha(Ki_{n,r})\le X^TL(Ki_{n, r})X<\alpha(G)$.
\end{proof}

 Now we are ready to prove Theorem~\ref{least-cliq}.

 {\textbf {Proof} } of  Theorem~\ref{least-cliq}.
 If $r=2$, then $Ki_{n, r}$ is  a path of order $n$ and the assertion follows from \cite{fiedler1973}. If $r=n-1$, the assertion follows from Lemma~\ref{switch2}. Hence we assume that $3\le r\le n-2$.   Let $ V=\{w_1, \dots, w_r, u_{1}, u_2,\dots, u_{n-r}\}$ be the vertex set of $G$ and the induced subgraph $G[w_1,\dots, w_r]$ is a clique of order $r$.  Assume that $G\neq Ki_{n, r}$.  We consider the following three cases.

 {\bf Case 1:} There exist two vertices, say $u_1, u_2$, in $\{u_1, \dots, u_{n-r}\}$
 and two vertices, say $w_1, w_2$, in $ \{w_1, \dots, w_r\}$  such that $w_1u_1, w_2u_2\in E(G)$.  Since the algebraic connectivity of a graph is nonincreasing function on deleting edges, we can delete as much as possible edges in $G$ excepting the edge set$\{w_1u_1, w_2u_2, w_iw_j, 1\le i\neq j\le r\}$ such that the resulted graph is still connected.  The resulted graph is denoted by $H$. Then
 $\alpha(G)\ge \alpha(H)$.  Without loss of generality, we assume that
 $w_1u_1,\dots, w_tu_t\in E(H)$ and there  are no other edges joining vertex set $\{w_{t+1}, \dots, w_r\}$ and $\{u_1, \dots, u_{n-r}\}$. Then the component $T_i$ of $G$ containing vertex $w_i$ by deleting edges $w_iw_j, j=1, \dots, r, j\neq i$ is a tree for $i=1, \dots, t$.
  Let  $d(w_i, u)=\max\{dist_{T_i}(w_i, v)\ |\  d_{T_i}(v)\ge 3, \ v\in V(T_1)\}$, where $d_{T_i}(v)$ is the degree of $v$ in $T_i$,  $dist_{T_i}(w_i, v)$ is the distance of between $w_i$ and $v$ in $T_i$. By the repeated use of Lemma~\ref{guo2}, we get a new graph $H_1$ such that $d_{T_i}(u)=2$ and $\alpha(H)\ge \alpha(H_1)$.  Further, by a series of the repeated use of Lemma~\ref{guo2}, there exists a graph $H_2$ such that $\alpha(H_1)\ge \alpha(H_2)$ and  the component of $H_2$ containing $w_i$ by deleting  edges $w_iw_j, j=1, \dots, r, j\neq i$ is a path starting $w_i$, $i=1, \dots, t$. Hence there exists a graph $H_3$ is the graph from $K_r$ and attached $t$ paths $P_1, \dots, P_t$ starting vertices $w_1, \dots, w_t$ and end vertices $u_{n-t-r-1}, \dots, u_{n-r}$, respectively.  If $t\ge 3$, then there exist two vertices, say $X(u_{n-r-1})X(u_{n-r})\ge 0$. by Lemma~\ref{guo1}, we get a new graph $H_4=H_3-w_tu_t+u_{n-1}u_t$ such that
  $\alpha(H_3)\ge \alpha(H_4)$. Hence by the repeated use of Lemma~\ref{guo1}, there exists a graph $H_5$ which is the graph obtained from $K_r$ with vertex set $\{w_1, \dots, w_r\}$ by attached a path $P$ of length $k$  at $w_1$ and a path $Q$ of length $l$ at $w_2$ such that
$\alpha(H_4)\ge \alpha(H_5)$ and $k\ge l\ge 1$.  Hence by Corollary~\ref{cor4.6}, $\alpha(G)\ge \alpha(H_5)>\alpha(Ki_{n, r})$.

{\bf Case 2:} There do not exist $1\le i\neq j\le r$ and $1\le p\neq q\le n-r$ such that
$w_iu_p, w_ju_q\in E(G)$ and there exist at least two vertices, say $w_1, w_2$, in $\{w_1, \dots, w_r\}$ and a vertex, say, $u_1$, in $\{u_1, \dots, u_{n-r}\}$ such that
$w_1u_1, w_2u_1\in E(G)$.  We are able to deleting as much as possible edges in $E(G)$ excepting edge set $\{w_iw_j, 1\le i\neq j\le r, w_1u_1, w_2u_1\}$ such that the resulted graph is still connected. The resulted graph is denoted by $H_6$ and $\alpha(G)\ge \alpha(H_6)$. Further, the component of $H_6$ containing $u_1$ from $H_6$ by deleting edges $w_1u_1,w_2u_1$. Then we are able to apply the repeated Lemma~\ref{guo2}, the final resulted graph $H_7$ is the graph of order $n$ from $K_r$  with vertex set $\{w_1, \dots, w_r\}$ and a path $P=u_1\dots u_{n-r}$ by joining two edges $w_1u_1$ and $w_2u_1$. Moreover $\alpha(H_6)\ge \alpha(H_7)$. By Lemma~\ref{switch2}, $\alpha(H_7)>\alpha(Ki_{n,r})$. So $\alpha(G)>\alpha(Ki_{n,r})$.

{\bf Case 3:} There exists only one vertex, say $w_1$, in $\{w_1, \dots, w_r\}$ such that it is adjacent to vertices in $\{u_1,\dots, u_{n-r}\}$, say $w_1u_1\in E(G)$.
 We are able to deleting as much as possible edges in $E(G)$ excepting edge set $\{w_iw_j, 1\le i\neq j\le r, w_1u_1\}$ such that the resulted graph is still connected.
The resulted graph is denoted by $H_8$ and $\alpha(G)\ge \alpha(H_8)$. By the repeated use of Lemma~\ref{guo2}, there exists a graph $H_9$  which is obtained from $K_r$ with vertex set $\{w_1,\dots,w_r\}$ and two paths $P=u_1\dots u_s$ and $Q=u_{s+1}\dots u_{n-r}$ by joining two edges $w_1u_1$ and $u_{s+1}u_i$ (or $u_{s+1}w_1$), $1\le i\le s-1$. It is easy to see that $X(u_s)\neq 0$ or $X(u_{n-r})\neq 0$ (otherwise, we are able to obtain $X(u_i)=0, i=1, \dots, n-r$, which implies $X=0$). By Lemma~\ref{guo2}, we obtain $\alpha(H_9)>\alpha(Ki_{n,r})$. We finish our proof.$\blacksquare$

\begin{corollary}\label{lower and upper}
Let $G$ be a graph of order $n$ with  clique number $r$. Then
$$\frac{n}{n-\alpha(G)}\le r\le n+1-\frac{4}{n\alpha(G)}.$$
\end{corollary}
\begin{proof}
The lower bound follows from Corollary~\ref{lower}. By \cite{mohar1991} and Theorem~\ref{least-cliq}, we have
$$\alpha(G)\ge \alpha(Ki_{n, r})\ge \frac{4}{n(n-r+1)},$$
which implies $r\le n+1-\frac{4}{n\alpha(G)}$.
\end{proof}

\end{document}